\documentclass[11pt]{article} \usepackage{amssymb}

\usepackage{amsfonts} \usepackage{amsmath} \usepackage{amsthm} \usepackage{bm}
\usepackage{latexsym} \usepackage{epsfig}
\usepackage{hyperref}

\usepackage{color}
\usepackage[usenames, dvipsnames, svgnames, table]{xcolor}
  
\newtheorem*{theorem*}{Theorem}
\newtheorem{theorem}{Theorem}[section]
\newtheorem{lemma}[theorem]{Lemma}
\newtheorem*{proposition*}{Proposition}

\newtheorem{proposition}[theorem]{Proposition}

\newtheorem{corollary}[theorem]{Corollary}

\newcommand{\ignore}[1]{}

\newcommand{\enote}[1]{} \newcommand{\knote}[1]{}
\newcommand{\rnote}[1]{}

\DeclareMathOperator{\argmax}{argmax}

\DeclareMathOperator{\Binom}{Binom}

\renewcommand{\P}{{\bf P}}

\newcommand{\E}{{\bf E}} 
\newcommand{\Var}{{\bf Var}}

\newcommand{\eps}{\epsilon}

\newcommand{\N}{\mathbb N} 
\newcommand{\Z}{\mathbb Z}

\newcommand{\CalA}{{\mathcal{A}}}

\newcommand{\CalH}{\mathcal{H}}

\newcommand{\half}{{\textstyle \frac12}}

\renewcommand{\phi}{\varphi}

\newcommand{\copyableTheorem}[2]{
 \newtheorem*{newthm#1}{Theorem \ref{thm#1}}
 \begin{newthm#1}
   {#2}
 \end{newthm#1}
 \expandafter\newcommand\expandafter{\csname thm#1\endcsname}{
   \begin{theorem}
     \label{thm#1}
     {#2}
   \end{theorem}
 }
}

\newcommand{\FullVersionOnly}[1]{{#1}}
\newcommand{\ShortVersionOnly}[1]{}

\begin{document}
\title{Majority Dynamics and Aggregation of Information in Social
  Networks} \author{Elchanan Mossel\footnote{Weizmann Institute and
    U.C. Berkeley. E-mail: mossel@stat.berkeley.edu. Supported by a
    Sloan fellowship in Mathematics, by BSF grant 2004105, by NSF
    Career Award (DMS 054829), by ONR award N00014-07-1-0506 and by
    ISF grant 1300/08}, Joe Neeman\footnote{UC Berkeley. E-mail:
    joeneeman@gmail.com.}~ and Omer Tamuz\footnote{Weizmann
    Institute. E-mail: omer.tamuz@weizmann.ac.il. Supported by ISF
    grant 1300/08. Omer Tamuz is a recipient of the Google Europe
    Fellowship in Social Computing, and this research is supported in
    part by this Google Fellowship.}}

\maketitle
\begin{abstract}
  Consider $n$ individuals who, by popular vote, choose among $q \geq
  2$ alternatives, one of which is ``better'' than the others.  Assume
  that each individual votes independently at random, and that the
  probability of voting for the better alternative is larger than the
  probability of voting for any other.  It follows from the law of
  large numbers that a plurality vote among the $n$ individuals would
  result in the correct outcome, with probability approaching one
  exponentially quickly as $n \to \infty$.

  Our interest in this paper is in a variant of the process above
  where, after forming their initial opinions, the voters update their
  decisions based on some interaction with their neighbors in a social
  network. Our main example is ``majority dynamics'', in which each
  voter adopts the most popular opinion among its friends. The
  interaction repeats for some number of rounds and is then followed
  by a population-wide plurality vote.

  The question we tackle is that of ``efficient aggregation of
  information'': in which cases is the better alternative chosen with
  probability approaching one as $n \to \infty$? Conversely, for which
  sequences of growing graphs does aggregation fail, so that the wrong
  alternative gets chosen with probability bounded away from zero?
  
  We construct a family of examples in which interaction prevents
  efficient aggregation of information, and give a condition on the
  social network which ensures that aggregation occurs.
  

  For the case of majority dynamics we also investigate the
  question of unanimity in the limit. In particular, if the voters'
  social network is an expander graph, we show that if the initial
  population is sufficiently biased towards a particular alternative
  then that alternative will eventually become the unanimous
  preference of the entire population.
\end{abstract}

\ShortVersionOnly{
\thispagestyle{empty}
\newpage
\setcounter{page}{1}
}

\section{Introduction}

The mathematical study of voting systems began as early as 1785, when
the Marquis de Condorcet\cite{MdC} observed what is essentially a
special case of the weak law of large numbers: suppose there is a
large population of voters, and each one independently votes
``correctly'' with probability $p > 1/2$. Then as the population size
grows, the probability that the outcome of a majority vote is
``correct'' converges to one.  Thus, information is ``efficiently
aggregated''. 

In this work, we study a simple model of voter interaction, in which
voters choose an independent random opinion initially, but then modify
that opinion iteratively, based on what their friends think. Thus
correlation between votes is introduced ``naturally'', through
interaction. Our main example of interaction is majority dynamics,
where at each round each voter adopts the opinion of the majority of
its neighbors.  The basic question that we address is that of
efficient information aggregation: for which modes of interaction is
information aggregated efficiently, and for which is it not?

Additionally, we study some conditions for the achievement of
unanimity, when the graph of social ties is an expander, and when
agents use majority dynamics.

\subsection{Model}
We consider an election in which a finite set $V$ of voters must
choose between $q \geq 2$ alternatives, which we will take to be the elements
of $[q] = \{0, 1, \dots, q-1\}$. The voters are connected by an
undirected social network graph $G=(V,E)$. Denote the neighbors of $v
\in V$ by $N_v$.

Each voter $v \in V$ will be initialized with a preference $X_v(0) \in
[q]$, picked independently from a distribution $\P$ over
$[q]$. 

At time $t \in \{1, \dots, T\}$, $v$ will update her opinion to
$X_v(t)$ based on what her friends' opinions at times $t-1$ and
earlier.  At time $T$, an election will take place and a winner $Y$
will be declared. Note that $Y$ is a deterministic function of the
initial votes $(X_v(0))_{v \in V}$.

A simple and important example is {\em majority dynamics} where $q=2$:
At each iteration of the dynamics, each individual $v$ sets her vote
to equal the most popular vote among her neighbors in the previous
iteration (we elaborate below on the handling of ties). 
\begin{align*}
  X_v(t) = \argmax_{a \in \{0,1\}}|\{w| X_w(t-1)=a,\; w \in N_v\}|.
\end{align*}
At some large time $T$ an election by plurality takes place, so that
the winner is
\begin{align*}
 Y = \argmax_{a \in \{0,1\}}|\{v |X_v(T)=a\}|. 
\end{align*}

Note that the majority rule (or more generally the plurality rule, in
the case of more than two alternatives) is {\em fair} and {\em
  monotone}: It is {\em fair} in the sense that is does not, as an
election system, treat one alternative differently than another; it is
invariant to a renaming of the alternatives. It is {\em monotone} in
the sense that having extra supporters cannot hurt an alternative's
case.

As a generalization of majority dynamics, we allow any updating of
opinions and any election system, provided that they are {\em fair}
and {\em monotone}. For example, an individual may give more weight to
some of her friends than to others, and the final election could an
Electoral College system. In Sec.~\ref{sec:threshold} we further relax
the fairness condition.

\subsection{Overview of the results}
\subsubsection{Social types}
Our study of information aggregation will utilize the idea of
\emph{social type}: we divide the voters $V$ into a partition $\CalA$
of social types, and ask that any two voters of the same social type
play the same r\^ole in the election process.  More precisely we
require that if the labels are removed from all individuals then it is
impossible to tell apart two individuals of the same social type. This
shall be rigorously defined in Section~\ref{sec:outline}.

In the case of majority dynamics, social types are induced by the
automorphisms of the graph $G$: $u,w \in V$ are of the same social
type if there exists an automorphism $\tau$ of $G$ such that
$\tau(u)=w$. Intuitively, this means that in an unlabeled drawing of
$G$ it is impossible to say which is $u$ and which is $w$; $u$ and $w$
are of the same social type if they play the same r\^ole in the
geometry of $G$, and hence play the same r\^ole in the election
process.

\subsubsection{Aggregation of information}
Without loss of generality, we will assume that alternative 0 is the
best alternative, and that the initial opinion of each voter is
slightly biased towards alternative 0: we take $X_v(0)$ to be a
multinomial random variable such that $\P(X_v(0) = 0) > \P(X_v(0) =
j)$ for any $j \ne 0$.

Although this bias could be very small, the law of large numbers
guarantees that with enough voters, the outcome of a plurality vote at
time 0 would choose the correct alternative, except with exponentially
low probability.  We refer to this property as {\em efficient
  aggregation of information}.

In Section~\ref{sec:aggregation}, we study if information is still
efficiently aggregated if we hold the vote at time $T$ instead, after
allowing the agents to interact.  One of our main results (stated
formally in Theorem~\ref{thmQuantitativeTransitive} below) is that
information is efficiently aggregated when each social type has many
members. In particular, we show that the probability of choosing the
correct alternative approaches one as the size of the smallest social
type approaches infinity, with a polynomial dependence.

This implies that in majority dynamics on a transitive graph, in which
case all voters are of the same social type, the outcome of the final
vote will be zero, except with probability that decreases polynomially
with the number of voters.

\subsubsection{Lack of Aggregation}

Perhaps surprisingly, the condition requiring increasing size of each
social type is necessary.  Indeed in Section~\ref{sec:agg-fails} we
provide an example with $q=2$ alternatives, majority dynamics and a
final majority vote, which results in the wrong outcome, with constant
probability regardless of the size of the population!


\subsubsection{Wider agreement, unanimity and expanders}

In Section~\ref{sec:threshold}, we ask when, following $T$ periods of
interaction, a large part of the population is in agreement.

Focusing on the case $q=2$ and majority dynamics, we show that the
proportion of the population that votes for alternative 0 at time $T$
is at least as large as the initial bias towards alternative
0. 

We push the agreement threshold to its extreme in
Section~\ref{sec:unanimity}, where we show that if the social network
is an expander graph, the mode of interaction is based on plurality,
and there is enough initial bias then eventually the entire population
will agree on alternative 0.

\subsection{Related work}
Our work is closely related to work of Kalai~\cite{Kalai:04} who
studies social choice using tools of discrete Fourier analysis.  Kalai
proves that any {\em binary} unbiased and monotone election system
aggregates information efficiently, given that all the voters have low
influence on the outcome.

Our work expands on this work in several directions: First, we
elucidate the role of voters types in this setup by showing that
having large number of voters of each type implies aggregation and
that without this condition aggregation may not occur. Second, we go
beyond the binary world and explore general outcome spaces. Finally
the questions of higher thresholds and unanimity were not considered
before.

Kanoria and Montanary~\cite{KM:trees} study majority dynamics with two
alternatives on regular (infinite) tree graphs, giving conditions
which lead to convergence to unanimity. Their work can also be
interpreted as a study of a zero temperature spin glasses, a model
also studied by Howard~\cite{howard2000zero} on 3-regular trees and
Fontes, Schonmann and Sidoravicius~\cite{fontes2002stretched} on
$\Z^d$.

Berger~\cite{berger2001dynamic} gives an example of a series of graphs
in which majority dynamics results in the adoption, by all
individuals, of the opinion of the individuals in a constant size
group, provided {\em they} all agree. Thus these graphs could serve in
place of our example (Section~\ref{sec:agg-fails}), showing how
aggregation fails when there is a small social type. We provide our
example for completeness, and because it is somewhat simpler.

Our work is related to the widely studied family of Gossip-based
protocols on networks (see, e.g., Bawa et
al.~\cite{bawa2003estimating}, Kempe et al.~\cite{kempe2003gossip},
and a survey by Shah~\cite{shah2009gossip}). The goal there is to
design and/or analyze distributed, repeated algorithms for the
aggregation of information on networks. For example, in the classical
DeGroot model~\cite{DeGroot:74} agents ``vote'' with a real number,
which they calculate at each iteration by averaging the votes of their
neighbors from the previous iteration. The agents all converge to the
same number, which is a good approximation of the average of the
initial votes only if degrees are low~\cite{golub2010naive}, or if,
indeed, the size of the smallest social type is large. This
model is fairly easy to analyze, since the votes in each iteration are
a linear function of the votes in the previous iteration. Majority
dynamics is a natural discretization of this process, but has proven
to be more resistant to analysis. Indeed the non-linearity of the
dynamics results not only in major technical challenges but also in
different behaviors of the two models.

Another related strain of models is that of Bayesian learning. Here
the agents optimize their votes to those which are the most likely to
be correct, given a prior over correct alternatives, an initial
private signal and the votes of their neighbors in previous rounds
(see, e.g.,~\cite{MosselSlyTamuz11:arxiv}). Perhaps surprisingly, this
dynamic is not necessarily monotone and therefore its analysis
requires different tools. The agents calculation there are more
complicated, and hence more difficult to analyze. On the other hand,
the optimality of the agents' actions makes the model amenable to
martingale arguments, which don't apply in the case of majority
dynamics.

Our main proof uses tools from the field of Fourier analysis of
Boolean functions on the discrete hypercube. 
In particular we use and extend results of Kahn, Kalai and
Linial~\cite{KaKaLi:88}, Friedgut and Kalai~\cite{FK:96}, a strong
version of the KKL theorem by Talagrand~\cite{Talagrand:94} and a
recent generalization by Kalai and Mossel\cite{KM:10}.

\subsection{Acknowledgments}
We would like to thank Miklos Racz for his careful reading of the manuscript
and his suggestions.

\section{Definitions and results}\label{sec:outline}
\subsection{Majority Dynamics}
Let $V$ be a finite set of individuals. Let $G=(V,E)$, an undirected
finite graph, represent the network of social connections of $V$. We
denote the neighbors of $v \in V$ by $N_v$. We allow $G$ to contain
self-loops, so that $v$ may or may not belong to $N_v$.

Let $X_v(t) \in \{0,1\}$ denote $v$'s vote at time $t \in
\{0,\ldots,T\}$. Let each $X_v(0)$ be chosen from some distribution
$\P$ over $\{0,1\}$, independently and identically for all $v \in
V$. Note that once the initial votes $(X_v(0))_{v \in V}$ are chosen,
the process is deterministic.

At times $t > 0$, $v$ updates its vote to equal the majority opinion
of its neighbors in the previous round. If the number of neighbors is
even then we either add or remove $v$ itself to the set of neighbors
$N_v$, to avoid ties.
\begin{align*}
  X_v(t) = \argmax_{a \in \{0,1\}}|\{w| X_w(t-1)=a,\; w \in N_v\}|.
\end{align*}
After some number of rounds $T$ an election by majority takes
place. We denote the winner by $Y_T$:
\begin{align*}
 Y_T = \argmax_{a \in \{0,1\}}|\{v |X_v(T)=a\}|. 
\end{align*}
To avoid ties in the final election, we assume $|V|$ is odd.

We next define {\em social types}. Recall that $\tau:V \to V$ is a
graph automorphism of $G=(V,E)$ if $(u,v) \in E \leftrightarrow
(\tau(u),\tau(v)) \in E$. We say that $u$ and $v$ are of the same
social type if there exists a graph automorphism that maps $u$ to
$v$. Informally, this means that $u$ and $v$ play the same r\^ole in
the geometry of the graph; it is impossible to tell which is which if
the labels are removed from the vertices. It is easy to see that
``being of the same social type'' is an equivalence relation. We
denote by $\CalA(G)$ the partition of the vertices of $G$ into social
types. We denote by $m(G)$ the size of the smallest social type:
\begin{align*}
  m(G) = \min_{A \in \CalA(G)} |A|.
\end{align*}

Our main result in this section is that information is aggregated
efficiently, provided that each social type has many members.  To
state our result, we first define the efficiency of an aggregation
procedure.  Let $\P_\delta$ be the probability distribution $\{0,1\}$
such that $\P_\delta(0) = \half(1+\delta)$ and $\P_\delta(1) =
\half(1-\delta)$. Then the \emph{efficiency} $\mu_{\delta}(G,T)$ of
majority dynamics on $G$ up to time $T$ is
\begin{equation*}
  \mu_\delta(G,T)=\P_\delta[Y_T = 0].
\end{equation*}
Note that in a slight abuse of notation we use $\P_\delta$ to denote both the
distribution over $\{0,1\}$ from which $X_v(0)$ is chosen, and the
measure on $(X_v(t))_{v\in V, 1\leq t \leq T}$ and $Y_T$ which is
induced by $\P_\delta$.

Our main result for this section is the following:
\begin{theorem*}
  There exists a universal constant $C > 0$ such that for any graph $G$
  \begin{align*}
    \mu_\delta(G,T) \ge 1 - C \exp\left(-C \frac{\delta \log m(G)}{\log (1/\delta)}\right).
  \end{align*}
\end{theorem*}
In particular, $\mu_\delta(G)$ approaches one as $m(G)$ tends to
infinity. Note that the bound does not depend on $T$. This theorem is
a special case of Theorem~\ref{thmQuantitativeTransitive}, which is
stated below.

In the other direction, we provide an example showing what can go
wrong when $m(G_n)$ does not grow to infinity.
\begin{theorem}\label{thm:example}
  For any $\delta > 0$, there exists a sequence of graphs $G_n$, whose
  sizes converge to infinity, such that
  \begin{align*}
    \sup_n \sup_{T \ge 1} \mu_\delta(G_n,T) < 1.
  \end{align*}
\end{theorem}
That is, there is some $\eps >0$ such that for any $n$ and $T$ the
probability of choosing the wrong alternative is at least $\eps$.

\subsection{Monotone Dynamics}
\label{sec:monotone-dynamics}
In this section we extend the definitions and results of the previous
section to a large class of update rules and election systems, and a
choice between more than two alternatives.

Let $[q] = \{0, 1, \dots, q-1\}$ be the set of alternatives. The
initial votes $X_v(0)$ are, as above, chosen i.i.d.\ from some $\P$,
which is now a distribution over $[q]$. As before, the process is
deterministic once the initial votes are chosen.

Let the history of $v$'s neighborhood before time $t$ be denoted by
$H_v(t) = (X_w(s))_{s < t, w \in N_v}$. Then $[q]^{[t] \times N_v}$ is
the set of possible histories of the neighborhood of $v$ before time
$t$.

For each $a \in [q]$ and $k \in \N$ we define a relation $\ge_a$ on
$[q]^k$ as follows.  Let $x, x' \in [q]^k$. We write $x' \ge_a x$ if,
for all $i \in [k]$ it holds that
\begin{align*}
  x'_i \neq x_i \rightarrow x'_i = a.    
\end{align*}
Alternatively, if a vector of votes $x$ is changed to $x'$ such that
$x' \ge_a x$, then for each $i$ either $x_i$ is unchanged, or it is
changed to $a$.  Note that when $q=2$ then $x' \ge_1 x$ reduces to the
usual $x' \ge x$, i.e., $x'_i \ge x_i$ for $i \in [k]$.

For $0 < t \leq T$, let $X_v(t)$ be determined as follows.  Let the
{\em mode of interaction} be a collection of functions $m_{v,t}:
[q]^{[t] \times N_v} \to [q]$, with
\begin{align*}
  X_v(t) = m_{v,t}(H_{v,t}). 
\end{align*}
These functions are generalization of the majority function used in
majority dynamics. As such, we require that they meet the following
conditions:
\begin{enumerate}
\item They are {\em fair}, or symmetric with respect to the
  alternatives: for all permutations $\sigma$ on $[q]$ and all
  histories $h \in [q]^{[t] \times N_v}$, $\sigma(m_{v,t}(h)) =
  m_{v,t}(\sigma(h))$, where $\sigma(h)$ is the result of applying
  $\sigma$ to each element of $h$.
\item They are {\em monotone}: for every pair $h, h' \in [q]^{[t]
    \times N_v}$, if $m_{v,t}(h) = a$ and $h' \ge_a h$ then
  $m_{v,t}(h') = a$.
\end{enumerate}
An example would be majority dynamics, i.e., the case where $q=2$,
$|N_v|$ is odd for all $v$, and $m_{v,t}$ is equal to the most popular
opinion among $X_w(t-1)$, where $w \in N_v$. A different simple
example is the case that $X_v(t)$ is simply equal to $X_v(t-1)$,
unless all of $v$'s neighbors agree in time $t-1$ on some alternative
$a$, in which case $X_v(t)=a$. That is, the agents do not change their
opinions unless their friends unanimously agree on a different
opinion.

Following $T$ rounds of interaction, we apply an {\em election system}
function $g: [q]^V \to [q]$ to $(X_1(T),\ldots,X_{|V|}(T))$, to
determine the chosen alternative $Y$:
\begin{align*}
  Y = g(X_1(T),\ldots,X_{|V|}(T)).
\end{align*}
This is again a generalization of a majority vote, and as such we
require that $g$ satisfy the same fairness and monotonicity
properties:
\begin{enumerate}
\item It is {\em fair}, or symmetric with respect to the alternatives:
  for all permutations $\sigma$ on $[q]$, $\sigma(g(a)) =
  g(\sigma(a))$, where $\sigma(a)$ is the result of applying $\sigma$
  to each element of $a$, $a_v$.
\item It is {\em monotone}: for every pair $x, x' \in [q]^V$, if $g(x) = a$
  and $x' \ge_a x$ then $g(x') = a$.
\end{enumerate}
Examples of such functions are the simple plurality function and
various recursive plurality (i.e., electoral college-like)
functions. Another important example is the dictator function, in
which $g(a) = a_v$, for some fixed $v$.

The whole process of social interaction and elections can be viewed as
a single function from the original signals $\{X_v(0)|v \in V\}$ to
$[q]$. We denote this function by $f:[q]^V \to [q]$, and call it the
{\em aggregation function}, so that
\begin{align*}
  Y = g(X_1(T),\ldots,X_{|V|}(T)) = f(X_1(0),\ldots,X_{|V|}(0)).
\end{align*}
Note that for brevity's sake we sometimes write the above as $Y=f(X)$.
It is easy to see that the aggregation function $f$ has the same
properties we that require from the election system $g$: it is
monotone and fair.

Finally, the concept of {\em social types} is, in the case of monotone
dynamics, related to the symmetries of the aggregation function,
rather than those of the graph. We first define $\CalH(f)$, the
symmetry group of the aggregation function, as the group of
permutations $\tau$ on $V$ that satisfy the following condition: for
every $a \in [q]^V$ it holds that $f(\tau(a)) = f(a)$, where
$\tau(a)_v = a_{\tau(v)}$.

It is easy to verify that $\CalH(f)$ is indeed a group, with
composition as the operation: for any $\tau,\sigma \in \CalH(f)$ it
holds that $f(\tau(\sigma(a))) = f(\tau(a)) = f(a)$, and hence
$\tau\sigma \in \CalH(f)$. Also, $f(a) = f(\tau(\tau^{-1}(a))) =
f(\tau^{-1}(a))$, and so $\tau^{-1}$ is also in $\CalH(f)$. 

The set of {\em Social types} is simply $V/\CalH(f)$, the set of
orbits of $V$ under the action of $\CalH(f)$.  I.e., $\CalA(f)$ is the
unique partition of $V$ such that $v,w \in V$ are of the same social
type iff $\exists \tau \in \CalH(f)$ such that $\tau(v) = w$.

The definition of $m(G,T)$ now naturally becomes the following.  Given
an aggregation function $f:[q]^V \to [q]$, denote by $m(f)$ the size
of the smallest social type induced by $f$:
\begin{align*}
  m(f) = \min_{A \in \CalA(f)} |A|.
\end{align*}

Our main result of this section, which is a strict generalization of
that of the previous, is again that information is aggregated
efficiently provided that each social type has many members.  In the
case of monotone dynamics, our definition of the efficiency of
aggregation is the following.  Let $\mathcal{P}_\delta$ be the set of
probability distributions $\P$ on $[q]$ under which $\P(0) \ge \P(i) +
\delta$ for all $i = 1, \dots, q-1$.  Then the \emph{efficiency}
$\mu_{\delta}(f)$ of a function $f: [q]^n \to [q]$ is defined by
\begin{equation*}
  \mu_\delta(f)=\inf_{\P \in \mathcal{P}_\delta} \P[f(X) = 0].
\end{equation*}

Our main result for this section is the following:
\copyableTheorem{QuantitativeTransitive}{ Let $f:[q]^V \to [q]$ be a
  monotone and fair aggregation function, and let $m=m(f)$ be the size
  of the smallest social type.  Then 
  \begin{align*}
    \mu_\delta(f) \ge 1 - C_q \exp\left(-C_q \frac{\delta \log m}{\log (1/\delta)}\right),
  \end{align*}
  for some $C_q$ that depends only on $q$.
}

Theorem~\ref{thmQuantitativeTransitive} is a statement about functions
$f$ such that $m(f)$ is large. For $q=2$ and odd $n$ it is easy to
find examples of such functions - the majority function, for example.
However, not for every value of $q$, $n$ and $m \leq n$ there exists a
fair and monotone aggregation function $f_{q,n}$ such that $m =
m(f_{q,n})$. In particular, it is not clear for which values of $q$
and $n$ there exists a fair and monotone aggregation function
$f_{q,n}$ that is {\em transitive}, i.e., $m(f_{q,n})=n$.

The challenge is to break ties in a way that preserves fairness and
transitivity, and indeed it seems that no simple, immediate examples
exist. We provide the following example of a fair, transitive and
monotone function, for any $q \geq 2$ and $n$ prime and larger than
$q$. See further discussion in~\cite{nima}.

\begin{proposition}
  For all $q \geq 2$ and $n$ prime and strictly larger than $q$, there
  exists a monotone, fair and transitive aggregation function $f :
  [q]^n \to [q]$.
\end{proposition}

\subsection{Unanimity results}

Here we consider any number of alternatives $q$, but specialize to the
case where the mode of interaction is given by simple plurality. It is
easy to construct examples showing that our earlier assumptions on the
structure of the network do not imply that the whole electorate will
eventually agree. Indeed, there could be a small clique of voters who
are well-connected to each other but poorly connected to the rest of
the population. These voters could forever maintain an opinion
contrary to that of their peers. One way to avoid this situation is to
ask that the social network be an expander graph.

Let $M$ be the adjacency matrix of a $d$-regular graph $G$. We say
that $G$ is a $\lambda$-expander graph if the second-largest absolute
eigenvalue of $M$ is no larger than $\lambda$.

Although we will not require any knowledge of expander graphs here, we
refer the uninitiated reader to~\cite{HLW:expanders} for a survey on
the topic. For now, it is enough to know that ``good'' expanders have
$\lambda = O(\sqrt d)$.

\begin{theorem}
  Let $G_n$ be a sequence of $d$-regular $\lambda$-expanders whose
  size converges to infinity.  Suppose that $\frac{\lambda}{d} \le
  \frac{3}{16}$ and
  \begin{equation}\label{eq:gap}
    \P(0) \ge \P(i) + \frac{c \sqrt {\log q}}{\sqrt d}
  \end{equation}
  for all $i \ne 0$.  For $v$ a vertex in $G_n$ let $X_v(t)$ be drawn
  i.i.d.\ from $\P$, and let the mode of interaction be majority
  dynamics. Then with probability converging to 1 as $n \to \infty$,
  there exists a time $T$ such that $X_v(T)=0$ for all $v \in V$.
\end{theorem}

The dependency on $d$ in Eq.~\eqref{eq:gap} is possibly not tight.  In
particular, if $q=2$ and the girth of $G_n$ tends to infinity with
$n$, then a result of Kanoria and Montanari~\cite{KM:trees} implies
that we can replace $\sqrt d$ by $d^\alpha$ for any $\alpha > 0$.

\subsection{Higher threshold results}
For $q=2$, consider the election system $g_\alpha(x) = 1(\sum_i x_i
\ge (1-\alpha) n)$.  When $\alpha = 1/2$, this is just the simple
majority function. It is monotone and symmetric and so
Theorem~\ref{thmQuantitativeTransitive} applies. When $\alpha > 1/2$, however,
$g_\alpha$ is no longer symmetric in the alternatives. We prove that
the final bias is as large as the expected initial bias.

\begin{theorem}\label{thm:threshold}
  Let $f_{n,\alpha}$ be a fair and monotone aggregation function with
  election system $g_{n,\alpha}$ on the graph $G_n$ after running $T$
  rounds of interaction. If $m(f_{n,\alpha}) \to \infty$ and $\alpha <
  \frac{1}{2} + \frac{\delta}{2}$ then for any $T \in \N$,
  \begin{align*}
    \lim_{n \to \infty} \mu_\delta(f_{n,\alpha}) = 1.    
  \end{align*}
\end{theorem}

We do not believe that the relationship between $\alpha$ and $\delta$
is the best possible. Note that for the complete graph on $n$ nodes,
one can take $\alpha$ exponentially close to $1$ for any $\delta$.  It
is natural to guess that the worst dependence on $n$ occurs in a
ring. For this case we show that one can take $\alpha$ as large as $1
- (1 - \delta)^2 / 2$.

\section{Aggregation of Information}\label{sec:aggregation}

In this section, we will prove the following theorem, using the
definitions of Section~\ref{sec:monotone-dynamics}.

\thmQuantitativeTransitive


The proof of this theorem relies on a ``sharp threshold'' theorem of
Kalai and Mossel~\cite{KM:10} (which is itself an extension of
Talagrand's theorem~\cite{Talagrand:94} to the case $q > 2$). Sharp
threshold theorems go back to Margulis~\cite{Margulis:74} and
Russo~\cite{Russo:81}; Friedgut-Kalai~\cite{FK:96} and
Kalai~\cite{Kalai:01} apply sharp threshold theorems in contexts
similar to this one. In fact, the result of~\cite{Kalai:01} gives a
weaker version of Theorem~\ref{thmQuantitativeTransitive} in which each social
type must have at least $n/o(\log n)$ members.

A crucial ingredient for sharp threshold results is the notion of
influence, which we will define for a function $f: [q]^n \to \{0,
1\}$.  Let $\P$ be a probability measure on $[q]$, and denote also by
$\P$ the corresponding product distribution over $[q]^n$.  The
influence of voter $i$ on a function $f: [q]^n \to \{0, 1\}$ is
\begin{equation}
  \label{eq:influence}
  I_\P^i(f) = \E_{\P} \Var_{\P}(f(X_1,\ldots,X_n) | X_1, \dots, X_{i-1}, X_{i+1}, \dots, X_n).
\end{equation}

Kalai and Mossel~\cite{KM:10} prove the following inequality:
\begin{theorem}\label{thm:kalai-mossel}
  Suppose that $\P(a) \ge \alpha > 0$ for every $a \in [q]$.  If
  $\max_i I_f^\P(i) \le \eps$ then
  \begin{align*}
    \sum_{i=1}^n I_\P^i(f) \ge C \log n \frac{\log (1/\eps) -
      \log(1/4)}{\log (1/\alpha)} \Var_{\P}(f)    
  \end{align*}
for a universal constant $C$.
\end{theorem}

Before proving Theorem~\ref{thmQuantitativeTransitive} we will require
a simple definition and Lemma.
Let $\P$ be a probability distribution on $[q]$ such that $\P(0) >
0$.  Define the following family of distributions $\P_t$ (indexed by
$t \in [0,1]$) as follows:
\begin{align*}
  \P_t(a) =
  \begin{cases}
    t & a=0\\
    (1-t)\P(a|a \neq 0) & a > 0.
  \end{cases}
\end{align*}  
Note that $\P_{\P(0)} = \P$. 
\begin{lemma}
  Let $\P$ be a probability distribution on $[q]$ such that $\P(0) =
  \P(1) + \delta$ for some $\delta > 0$. Let $s$ be such that $\P_s(0)
  = \P_s(1)$. Then
  \begin{equation}\label{eq:gap-ineq}
    \P(0) - s \ge \delta / 2.
  \end{equation}
\end{lemma}
\begin{proof}
  We can solve for $s$ to find that $s = (\P(0) -
  \delta)/(1-\delta)$. Hence
  \begin{align*}
    \P(0) - s = (1 - \P(0)) \frac{\delta}{1-\delta} \geq \delta / 2,
  \end{align*}
  Where the inequality follows from the fact that since $\P(0)
  = \P(1) + \delta \le 1 - \P(0) + \delta$, it holds that $1-\delta
  \le 2-2\P(0)$. 
\end{proof}
We prove Theorem~\ref{thmQuantitativeTransitive} below by calculating
the derivative of $\P_t(f=0)$ with respect to $t$ and then integrating
between $t = s$ and $t = \P(0)$. We thus interpolate between $\P_s$,
in which the probability of $0$ and $a$ are equal, and $\P$
($=\P_{\P(0)}$), in which the probability of $0$ is larger by $\delta$
than the probability of $a$.

For a function $g$ and a probability measure $\P$, we will write
$\P(g)$ for the expectation of $g$ under $\P$.

\begin{proof}[Proof of Theorem~\ref{thmQuantitativeTransitive}]
  Since the conclusion of the theorem is only weakened when $\delta$
  is reduced, we can assume without loss of generality that the
  inequality $\P(0) \geq \P(i) + \delta$ is tight and that $\P(1) +
  \delta = \P(0)$.  Choose $s \in [0, \P(0)]$ so that $\P_s(0) =
  \P_s(1)$.

  Define $g = 1_{(f = 0)}$.  Suppose (for now) that $\P(b) \ge
  \frac{\delta}{2q}$ for all $b \in [q]$, and so $\P_t(b) \ge
  \frac{\delta}{2q}$ for all $s \le t \le \P(0)$ and all $b \in [q]$.
  Since $f$ is fair and monotone, $\P_s(g) \ge 1/q$. Using
  monotonicity again, $\P_t(g) \ge 1/q$ for all $t \ge s$.

  By Theorem~\ref{thm:kalai-mossel}, if $\eps_t := \max_{i \in [n]}
  I_{\P_t}^i(g) < 1/10$ for all $i$ then
  \begin{align*}
    \sum_{i=1}^n I_{\P_t}^i (g) \ge
    \frac{C \log(1/\eps_t)}{\log (2q/\delta)} \Var_{\P_t}(g) \ge
    \frac{C \log (1/\eps_t)}{q \log (2q/\delta)} \P_t(1 - g)
  \end{align*}
  for all $t \in [s, \P(0)]$. Now, recall that for $A \in \CalA$, if
  $i,j \in A$ then they play the same r\^ole in $f$ and in particular
  have the same influence. Hence $\sum_{i=1} I_{\P_t}^i (g) \ge m
  \eps_t$, since $|A| \ge m$ for any $A \in \CalA$. In particular, if
  $\eps_t \ge (\log m)/m$ then $\sum_i I_{\P_t}^i (g) \ge \log m$;
  on the other hand, if $\eps_t \le (\log m)/m$ then the display above
  implies that 
  \begin{align}
    \label{eq:I-ineq}
    \sum_i I_{\P_t}(g) \ge C_q \frac{\log m} { \log (1/\delta)} \P_t(1-g),
  \end{align}
  for some $C_q$ that depends only on $q$.  This last inequality
  (Eq.~\ref{eq:I-ineq}) holds, therefore, in either case.

  On the other hand, Lemma 2.3 of~\cite{KM:10} (a generalization of
  Russo's formula) gives
  \[
  \frac{\partial \P_t(g)}{\partial t} \ge \sum_{i=1} I_{\P_t}^i(g)
  \]
  and so
  \[
  \frac{\partial \P_t(g)}{\partial t} \ge C_q \frac{\log m}{\log (1/\delta)} \P_t(1 - g)
  \]
  for all $t \in [s, \P(0)]$. Integrating between $s$ and $t$, we have
  \[
  \P_t(g) \ge 1 - \frac{1}{q} \exp\left(-C_q \frac{\log m}{\log (1/\delta)}(t - s)\right)
  \]
  and so we conclude by setting $t = \P(0)$ and invoking
  Eq.~\eqref{eq:gap-ineq}.

  Now, if the hypothesis $\P(b) \ge \frac{\delta}{2q}$ fails then we
  construct $\tilde \P$ by $\tilde \P(0) = \P(0) - \delta/2$ and
  $\tilde \P(b) = \P(b) + \frac{\delta}{2(q-1)}$ for $b \ne
  0$. Setting $\tilde \delta = \delta / 2$, we see that $\tilde \P$
  satisfies the hypothesis of the theorem (with $\delta$ replaced by
  $\tilde \delta$) and it also satisfies $\tilde \P(b) \ge
  \frac{\tilde \delta}{2q}$.  The proof goes through, then, and we can
  absorb the extra factor of 2 into the constant $C_q$.
\end{proof}

\subsection{Where aggregation fails}
\label{sec:agg-fails}
Let $q = 2$ and suppose that both the interaction mode and the
election system are given by simple majority votes. In this scenario,
we prove Theorem~\ref{thm:example} by giving an example with two
social types, one of which has a constant size as $n \to
\infty$. Information will not aggregate asymptotically in this
example, and the reason for the failure will be the presence of the
constant-sized social type.

Since $q=2$, it will be more convenient to set $p = \frac{1}{2} +
\frac{\delta}{2} = \P(0)$ and to write our example in terms of $p$
instead of in terms of $\delta$.  Let $G_n=(A\cup B, E)$, where
$|A|=1/(1-p)$ and $|B|=n(1/(1-p)+1)$. Then in particular the number of
vertices in $G_n$ is at least $n$. We assume here that $1/(1-p)$ is an
integer.

Let each $a\in A$ be connected to each $b \in B$, and let none of the
vertices in $A$ be connected to each other. The vertices in $B$ are
arranged in $n$ cliques, each of size $1/(1-p)+1$, and there are no
edges between the cliques. Each vertex in $B$ has a self-loop.

The degree of the vertices in $B$ is odd, since each has edges to
$2/(1-p)+1$ edges. To make the degrees in $A$ odd add a vertex that is
connected to all vertices in $A$. An isolated vertex can be added to
make the total number of vertices odd.

Henceforth we condition on the event that $X_v(0)=1$ for all $v \in
A$. Note that this happens with probability
$(1-p)^{|A|}=(1-p)^{1/(1-p)}$.

Let $C$ be one of the cliques of $B$. If at least one vertex $w$ in
$C$ votes 1 initially (at time $t=0$) then {\bf all} the vertices in
$C$ will vote 1 in the next round ($t=1$); each will have at least
$1/(1-p)+1$ neighbors ($\{w\}\cup A$) that vote 1 and at most
$1/(1-p)$ neighbors ($B \setminus \{w\}$) that vote 0. The probability
that at least one vertex in $C$ votes 1 initially is $1-p^{1/(1-p)}$,
which is greater than $1-1/e$, or about 0.63. Hence the number of
cliques in which all vertices will vote 1 at time 1 will be
distributed $\Binom\left(n, 1-p^{1/(1-p)}\right)$, which dominates the
distribution $\Binom\left(n, 0.6\right)$.

By Hoeffding's inequality, the probability that a majority of the
cliques (and hence a majority of the vertices) will vote 1 at time 1
is at least $1-\exp(-0.02n)$. Once this happens, the vertices in $A$
will all vote 1 in all future iterations, and so will these
cliques. Hence for all $T \geq 2$ a majority vote will result in 1.

The event that a majority of the cliques have a voter that initially
votes 1 is independent of the event that all vertices in $A$ initially
vote 1. Hence both events happen with probability at least
$(1-p)^{1/(1-p)}(1-\exp(-0.02))$. Since this quantity is positive and
independent of $n$, it follows that information does not aggregate and
Theorem~\ref{thm:example} is proved.

Berger~\cite{berger2001dynamic} constructs an example of a family of
graphs with $n$ vertices. In each graph there exists a set of at most
18 vertices (which he calls a {\em dynamic monopoly}), such that if
all agents in this set initially vote identically then, in majority
dynamics with two alternatives, all the agents converge to the initial
vote of the dynamic monopoly. In particular, this implies that in this
example, with probability at least $(1-p)^{18}$, aggregation fails for
any $n$. This is another example of how aggregation can fail when a
particular social type has a small size (in this case at most 18).

\ShortVersionOnly{\bibliographystyle{abbrv} \bibliography{majority}\appendix}

\section{The existence of monotone, fair and transitive aggregation
  functions}

\begin{proposition}
  For all $q \geq 2$ and $n$ prime and strictly larger than $q$, there
  exists a monotone, fair and transitive aggregation function $f :
  [q]^n \to [q]$.
\end{proposition}
\begin{proof}
  Let $q \geq 2$ and let $n>q$ be prime. Let $f : [q]^n \to [q]$ be
  defined as follows.

  For $a=(a_0,\ldots,a_{n-1}) \in [q]^n$ let $Q(a)$ be the set of
  alternatives that received the most votes. If $Q(a)=\{b\}$ is a
  singleton then let $f(a) = b$. Otherwise $|Q(a)| \geq 2$.  Let $M(a)
  \subset [n]$ be the set of voters that voted for one of
  the alternatives in $Q(a)$. Note that $|M(a)| \neq n$, since otherwise
  each alternative received the same number of votes and so $|Q(a)|$
  divides $n$, which is impossible since $n$ is prime. Also, $M(a)$ is
  clearly not the empty set, and so $|M(a)|$ is an invertible element
  of the field $\Z_n$. Let
  \begin{align*}
    k(a) = \frac{1}{|M(a)|}\sum_{i \in M(a)}i = \frac{1}{|M(a)|}\sum_{a_i \in Q(a)}i
  \end{align*}
  where addition and division are taken over the field $\Z_n$. Note
  that $k(a)$ is the ``average'' position of a voter that voted for
  one of the votes that received the most votes. Let
  \begin{align*}
    \ell(a) = \min\{0 \leq i < n: k(a) + i \in M(a)\},
  \end{align*}
  where again the sum $k(a) + i$ is taken over $\Z_n$. Finally, define
  \begin{align*}
    f(a) = a_{k(a) + \ell(a)}.
  \end{align*}

  By definition $f(a) \in Q(a)$, and so $f$ is the plurality function
  with some tie breaking rule, and is therefore monotone. Also, none
  of the alternative names appear in its definition, and it is
  therefore fair. It remains to show that it is transitive. We do this
  by showing that for each $0 \leq i_1 \leq i_2 < n$ there exists a
  permutation $\tau=\tau_{i_1,i_2}$ on $[n]$ such that $\tau(i_1) =
  i_2$ and $f(\tau(a)) = f(a)$, where $\tau(a) =
  (a_{\tau(0)},\ldots,a_{\tau(n-1)})$.

  Let $\tau_{i_1,i_2}(i) = \tau(i) = i - i_1 + i_2 \mod n$. Note that
  $Q(\tau(a)) = Q(a)$ and that $M(\tau(a)) = \tau^{-1}(M(a))$, so that
  $|M(\tau(a))| = |M(a)|$. Hence
  \begin{align*}
    k(\tau(a)) &= \frac{1}{|M(\tau(a))|}\sum_{i  \in M(\tau(a))}i\\
    &= \frac{1}{|M(a)|}\sum_{i \in \tau^{-1}(M(a))}i.
  \end{align*}
  By a change of variables we get that
    \begin{align*}
    k(\tau(a)) &= \frac{1}{|M(a)|}\sum_{i \in M(a)}\tau^{-1}(i)\\
    &= k(a) + i_1 - i_2\\
    &= \tau^{-1}(k(a))
  \end{align*}
  Next, 
  \begin{align*}
    \ell(\tau(a)) &= \min\{0 \leq i < n: k(\tau(a)) + i \in
    M(\tau(a))\}\\
    &= \min\{0 \leq i < n: k(a) + i_1 - i_2  + i\in M(a) + i_1 -i_2\}\\
    &= \ell(a),
  \end{align*}
  and finally, since $\tau(i+j) = \tau(i)+j$:
  \begin{align*}
    f(\tau(a)) = a_{\tau\big(k(\tau(a)) + \ell(\tau(a))\big)}
    = a_{\tau\big(k(\tau(a))\big) + \ell(a)}
    = a_{k(a) + \ell(a)} = f(a).
  \end{align*}

\end{proof}

\section{On higher thresholds of agreement}\label{sec:threshold}

In this section we again specialize to the case of $q=2$ alternatives,
and consider the question of when it can be shown that, after a number
of rounds of fair and monotone dynamics, a large proportion of the
population will agree on the correct alternative.

Consider the election system $g_\alpha(x) = 1(\sum_i x_i \ge
(1-\alpha) n)$.  When $\alpha = 1/2$, this is simply the majority
function, and so our earlier results apply, and under the appropriate
conditions $Y = g(X_1(T),\ldots,X_{|V|}(T))$ will equal $0$ with high
probability. What about when $\alpha > 1/2$? In this case $Y$ will
equal $0$ only if an $\alpha$ fraction of the population votes $0$ at
time $T$. When does this happen with high probability?

Since $g_\alpha$ satisfies the same transitivity properties as
$g_{1/2}$, the proof of Theorem~\ref{thmQuantitativeTransitive} mostly still
applies. At least, the ``sharp threshold'' part of the claim is still
true: there is some $p^* \in (0, 1)$ such that $\P(0) > p^*$ implies
that $\P(Y = 0) \to_{m(f_n)} 1$. Since $g_{\alpha}$ is no longer
anti-symmetric, however, we no longer know that the threshold occurs
at $p^* = 1/2$.

In this section, we will show that $p^* \le \alpha$, but we will also
give a simple example for which $p^* = 1 - O((1-\alpha)^2)$ as $\alpha \to 1$.
Thus, there may be a large gap between our bound and the true behavior
of $p^*$.

The first step is to obtain a lower bound on $\E \sum_i X_v(t)$.  The
argument here appeared in a course taught by the first author in Fall
2010, although it may have been known before then. In any case, we
give a proof for completeness.  For the rest of this section, $\P_p$
denotes the probability distribution on $\{0, 1\}$ satisfying $\P_p(0)
= p$, in which case $\delta = 2p-1$. As above, we also denote by
$\P_p$ the distribution over $n$ i.i.d.\ random variables distributed
$\P_p$.

\begin{lemma}
Let $f: \{0, 1\}^n \to \{0, 1\}$ be a monotone
function with $\P_{1/2}(f = 0) \ge \frac{1}{2}$.
Then $\P_p(f=0) \ge p$ for all $p \in [\frac{1}{2}, 1]$.
\end{lemma}

Note that equality holds for the function $f(x) = x_i$. In other
words, every monotone function aggregates information at least as well
as a dictator function. It is easy to construct less pathological
examples that come arbitrarily close to achieving this bound.

\begin{proof}
By the chain rule,
\begin{align*}
\frac{\partial \P_p(f)}{\partial p}
&= \sum_{i=1}^n \P_p(f(X_1, \dots, X_{i-1}, 0, X_{i+1}, \dots, X_n)
- f(X_1, \dots, X_{i-1}, 1, X_{i+1}, \dots, X_n)) \\
&= -\frac{1}{p(1-p)}\sum_{i=1}^n I^i_{P_p}(f).
\end{align*}
By the Efron-Stein inequality, $\sum I^i (f) \ge \Var(f)$,
with equality only if $f$ depends just on
one coordinate. If $f$ depends just on one coordinate, then the
proof is trivial, so we can suppose the contrary. Thus
$\frac{\partial}{\partial p} \P_p(f) < -\frac{1}{p(1-p)}\Var_{\P_p}(f)$.

Suppose, for a contradiction, that $1 - \P_p(f) = \P_p(f=0) < p$ for some
$p > \frac{1}{2}$.
Let $r$ be the infinum over all $p$ satisfying the previous sentence.
Since $\P_p(f)$ is a smooth function of $p$, it follows that
$\P_r(f) = 1-r$ and so $\Var_{\P_r}(f) = r(1-r)$.
Thus, $\frac{\partial}{\partial p} \P_p(f)|_{p=r} < -1$,
contradicting the assumption that $\P_p(f) > 1-p$ for
arbitrarily close $p > r$.
\end{proof}

Note that for any vertex $v$ and any $t$,
the conditions of the lemma hold for $f = X_v(t)$. Summing over
all $v$, we obtain the following:

\begin{corollary}\label{cor:in-expectation}
Suppose that $X_v(0)$ are independent Bernoulli variables with
mean $p \ge \frac{1}{2}$. Then, for any $t$,
\[
\E \sum_{v \in V} X_v(t) \le (1-p)|V|.
\]
\end{corollary}

Combining this with the proof of
Theorem~\ref{thmQuantitativeTransitive}, we arrive 
at the promised bound on the location of the sharp threshold.
Of course, this is just a restatement of Theorem~\ref{thm:threshold}.

\begin{corollary}\label{cor:threshold}
  Let $f_n:[q]^n \to q$ be a sequence of aggregation functions with
  monotone and fair modes of interaction and election system
  $g_{\alpha}$ as defined above.  Suppose that $\lim_{n \to \infty}
  m(f_n) = \infty$, and that $p > \alpha$. Then $\P_p(Y = 0)
  \to 1$ as $n \to \infty$.
\end{corollary}

\begin{proof}
  For the sake of brevity, denote $g_\eps =
  g_\eps(X_1(T),\ldots,X_{|V|}(T))$ (which equals $Y$ for
  $\eps=\alpha$). From the proof of Theorem~\ref{thmQuantitativeTransitive}, we
  have
\[
\frac{\partial \P_p(g_{\alpha} = 0)}{\partial p}
\ge
C (\log m) \Var_{\P_p}(g_{\alpha}).
\]
On the other hand, Corollary~\ref{cor:in-expectation} gives us that
for any $\eps>0$
\[
\P_p(g_{p - \eps} = 0) = \P_p \left(\sum_{v \in V} X_v(t) \le (1 - p + \eps)|V|\right) \ge \eps
\]
and so
$\Var_{\P_p}(g_{p-\eps}) \ge \eps \P_p(g_{p-\eps})$
for any $\eps$.

Fix $\alpha < p$ and set $\eps = (p-\alpha) / 2$.
Then for any $r \in [\alpha + \eps, p]$,
$\Var_{\P_r}(g_\alpha) \ge \eps \P_r(g_{\alpha})$
and so
we can solve the differential inequality
\[
\frac{\partial \P_r(g_{\alpha} = 0)}{\partial r}
\ge C \eps (\log m) \P_r(g_\alpha)
\]
in the range $[p - \eps, p]$, with initial
condition $\P_{p-\eps} (g_\alpha = 0) \ge \eps$.
We obtain
\[
\P_p(Y=0) = \P_p (g_\alpha=0) \ge 1 - (1-\eps) \exp\left(-C\eps^2 \log m) \right)
\]
and we send $m \to \infty$.
\end{proof}

\subsection{An example: cycles}

Let $G_n$ be a cycle on $n$ vertices, where each vertex has a
self-loop, and recall that $p = \frac{1 + \delta}{2}$.  When the mode of
interaction is majority dynamics, we can explicitly
calculate the distribution
of $\lim_{t\to\infty} X_v(t)$. This
will yield a wider bound (compared to Theorem~\ref{thm:threshold})
on the range of $\alpha$ for which
$\lim_{n\to \infty}\mu_\delta(f_{n,\alpha}) \to 1$. Of particular
interest are the cases when $\delta \to 0$ or $\delta \to 1$;
for small $\delta$,
$\alpha < \frac{1}{2} + \frac{5}{6} \delta - \Omega(\delta^3)$
turns out to imply
$\lim_{n\to\infty} \mu_\delta(f_{n,\alpha}) \to 1$, while
for large $\delta$, if we set
$\epsilon = 1-\delta$, then
$\alpha < 1 - \frac{1}{2} \epsilon^2$ is sufficient.
Therefore, the bound in
Theorem~\ref{thm:threshold} is not tight: for $\delta$ close to zero,
one can take $\alpha \approx \frac{1}{2} + \frac{5}{6} \delta$
while Theorem~\ref{thm:threshold} only guarantees
that $\alpha = \frac{1}{2} + \frac{1}{2} \delta$ will work; for
$\delta$ close to 1, $\alpha \approx 1 - \frac{1}{2} \epsilon^2$
is sufficient, but Theorem~\ref{thm:threshold} only
gives $\alpha = 1 - \frac{1}{2} \epsilon$.

The analysis of the cycle is relatively simple because the eventual
state of the voters can be easily foretold from the initial
state. First of all, whenever two (or more) adjacent voters share the
same opinion, they will retain that opinion forever. Moreover, strings
of voters whose opinions alternate will gradually turn into strings
of voters with the same opinion, as in the following example:
\[
  \begin{matrix}
    \text{time } t & \cdots & 1 & 1 & 0 & 1 & 0 & 1 & 0 & 0 & \cdots \\
    \text{time } t + 1& \cdots & 1 & 1 & 1 & 0 & 1 & 0 & 0 & 0 & \cdots \\
    \text{time } t + 2& \cdots & 1 & 1 & 1 & 1 & 0 & 0 & 0 & 0 & \cdots 
  \end{matrix}
\]
In fact, one can tell the eventual opinion of a voter $v$ with the
following simple rule: let $V \ge 0$ be the smallest number such that
$X_{v-V} = X_{v-V-1}$ and let $W \ge 0$ be the smallest number such that
$X_{v+W} = X_{v+W+1}$ (assuming that such $V$ and $W$ exist, which will
only fail to happen in the unlikely event that the whole
cycle consists of alternating opinions).
If $V \le W$ then $X_v(t) = X_{v-V}(0)$
for all $t \ge V$.
On the other hand, if $W \le V$ then
$X_v(t) = X_{v+W}(0)$ for all $t \ge W$.
(If $V = W$ then $X_{v -V}(0) = X_{v + W}(0)$ because
$X_{v-V}(0) = X_v(0)$ if and only if $V$ is even, and similarly
for $W$.)

\begin{proposition}
For any $v$,
\[
  \lim_{T\to\infty} \lim_{n \to \infty} \P(X_v(t) = 0 \text{ for all }
  t \ge T) =
\frac{2p^2 - p^3}{1 - p + p^2}
= \frac{1}{2} + \frac{5\delta - \delta^3}{6 + 2\delta^2}
= 1 - \frac{4 \epsilon^2 - \epsilon^3}{8 - 4 \epsilon + 2\epsilon^2}.
\]
\end{proposition}

As we observed following Corollary~\ref{cor:in-expectation},
this implies that if
$\alpha < \frac{1}{2} + \frac{5\delta - \delta^3}{6 + 2\delta^2}$
and the number of interaction rounds is sufficiently large
(depending on $\alpha$ and $p$), then
$\mu_\delta(f_{n,\alpha}) \to 1$.

\begin{proof}
For brevity, we will write $X_v$ instead of $X_v(0)$ for the initial
state of vertex $v$.
Instead of majority dynamics on the cycle, consider majority dynamics on
$\Z$; we will see later that these are essentially the same
when $n$ is large.
We may assume without loss of generality that $v = 0$.
As in the discussion above, let $V \ge 0$ be minimal such that
$X_{-V} = X_{-V-1}$
and let $W \ge 0$ be minimal such that $X_W = X_{W+1}$.

Let us first condition on $X_0(0) = 0$.
Consider the i.i.d.\ sequence
\[ Y_k = (X_{-2k}, X_{1-2k}, X_{2k-1}, X_{2k}) \in \{0, 1\}^4. \]
If $Y_1, \dots, Y_j = (0, 1, 1, 0)$ then the sequence
$X_{-2j}, \dots, X_{2j}$ consists of alternating zeros and ones, 
and so $V, W \ge 2j$. Define $A_0, A_1 \subset \{0, 1\}^4$
by
\begin{align*}
A_0 &= \{(a, b, c, d): b = 0 \text{ or } c = 0\} \\
A_1 &= \{(a, b, c, d): a = b = c = 1 \text{ or } b = c = d = 1\}
\end{align*}
Note that $A_0 \cap A_1 = \emptyset$ and
$\{0, 1\}^4 \setminus (A_0 \cup A_1) = \{(0, 1, 1, 0)\}$.
Therefore, if $J$ is minimal such that
$Y_J \ne (0, 1, 1, 0)$ then $Y_J$ is in either $A_0$ or $A_1$.
If $Y_J \in A_0$ then either $W = 2J-2$ and $X_W = 0$:
\[
  \begin{matrix}
    X_0 & X_1 & X_2 & X_3 & X_4 & \cdots & X_{2J-3} & X_{2J-2} & X_{2J-1} \\
    0   & 1   & 0   & 1 & 0 & \cdots & 1 & 0 & 0 \\
  \end{matrix}
\]
or $V = 2J-2$
and $X_{-V} = 0$:
\[
  \begin{matrix}
    X_{-(2J -1)} & X_{-(2J-2)} & X_{-(2J-3)} & \cdots & X_{-4} & X_{-3}
    & X_{-2} & X_{-1} & X_0 \\
    0 & 0 & 1 & \cdots & 0 & 1 & 0 & 1 & 0
  \end{matrix}
\]
In either of these cases, $X_0(t) = 0$ for all
$t \ge 2J-2$.
Conversely, if $Y_J \in A_1$ then either $X_W = 1$ or $X_V = 1$
and $X_0(t) = 1$ for all $t \ge 2J-1$.
Thus, (using the fact that $J$ and $Y_J$ are independent)
\begin{equation}\label{eq:cycle-example-1}
  \P(X_0(t) = 0 \text{ for all $t \ge T$ }| X_0 = 0) = \P(Y_J \in A_0)
\P(2J-2 \le T).
\end{equation}
Since the $Y_j$ are i.i.d,
\begin{equation}\label{eq:cycle-example-2}
  \P(Y_J \in A_0) = \frac{\P(Y_1 \in A_0)}{\P(Y_1 \in A_0 \cup A_1)}
= \frac{2p - p^2}{2p - p^2 + 2(1-p)^3 - (1-p)^4}
= \frac{2p - p^2}{1 -  p^{2} + 2 p^{3} -  p^{4}},
\end{equation}
where we have computed $\P(Y_1 \in A_i)$ by the inclusion/exclusion
formulas
\begin{align*}
  \P(Y_1 \in A_0) &= \P(X_{-1} = 0) + \P(X_1 = 0) - \P(X_{-1} = X_1 = 0) \\
  \P(Y_1 \in A_1) &= \P(X_{-2} = X_{-1} = X_1 = 1)
  + \P(X_{-1} = X_1 = X_2 = 1)
  - \P(X_{-2} = \cdots = X_2 = 1).
\end{align*}

The case for $X_0 = 1$ is similar: we define
\begin{align*}
A'_0 &= \{(a, b, c, d): a = b = c = 0 \text{ or } b = c = d = 0\} \\
A'_1 &= \{(a, b, c, d): b = 1 \text{ or } c = 1\}.
\end{align*}
If $J'$ is minimal such that $Y_{J'} \ne (1, 0, 0, 1)$ then
$Y_{J'} \in A'_0$ implies $X_0(t) = 0$ for $t \ge 2J' - 1$, while
$Y_{J'} \in A'_1$ implies $X_0(t) \to 1$ for $t \ge 2J' - 2$. Since
$\P(Y_1 \in A'_0) = 2p^3 - p^4$ and $\P(Y_1 \in A'_1) =
2(1-p) - (1-p)^2$, we have
\begin{equation}\label{eq:cycle-example-3}
  \frac{\P(X_0(t) = 0 \text{ for all } t \ge T | X_0 = 1)}
  {\P(2J-1 \le T)}
  = \frac{\P(Y_1 \in A'_0)}{\P(Y_1 \in A'_0 \cap A'_1)}
  = \frac{2p^3 - p^4}{1 - p^2 + 2p^3 - p^4}.
\end{equation}

To transition back from dynamics on $\Z$ to dynamics on the $n$-cycle,
note that the event $\{X_0(t) = 0 \text{ for all } t \ge T\}$ is the
same event on $\Z$ and on the $n$-cycle, provided that $n > 2T$.
In particular,~\eqref{eq:cycle-example-1} and~\eqref{eq:cycle-example-2}
imply that
\[
  \lim_{n \to \infty}\P(X_0(t) = 0 \text{ for all $t \ge T$ }| X_0 = 0)
  = \P(2J-2 \le T) \frac{2p - p^2}{1 -  p^{2} + 2 p^{3} -  p^{4}}
\]
for majority dynamics on the $n$-cycle
(and similarly conditioned on $X_0 = 1$, using~\eqref{eq:cycle-example-3}.
Since $\lim_{T \to \infty} \P(2J-2 \le T) = 1$,
\[
  \lim_{T \to \infty}
  \lim_{n \to \infty}\P(X_0(t) = 0 \text{ for all $t \ge T$ }| X_0 = 0)
  = \frac{2p - p^2}{1 -  p^{2} + 2 p^{3} -  p^{4}}.
\]
(and similarly conditioned on $X_0 = 1$). Finally,
\begin{align*}
  &\P(X_0(t) = 0 \text{ for all } t \ge T) \\
  &=
  p \P(X_0(t) = 0 \text{ for all } t \ge T | X_0 = 0)
  + (1-p) \P(X_0(t) = 0 \text{ for all } t \ge T | X_0 = 1)
  &\to \frac{2p^2 - p^3}{1 - p + p^2}
\end{align*}
as $T, n \to \infty$.
The formulas in terms of $\delta$ and $\epsilon$ are obtained by
substituting $p = \frac{1 + \delta}{2} = 1 - \frac{\epsilon}{2}$.
\end{proof}

\section{Expander graphs converge to unanimity}
\label{sec:unanimity}

\subsection{Majority dynamics with two alternatives}
\label{sec:majority-expanders}
In this section we again consider the case that $q=2$ and majority
dynamics (i.e., each voter adopts the majority opinion of its
neighbors), with a population wide majority vote at time $T$. To avoid
the issue of ties, we assume that $|N_v|$ is odd for all $v$ and that
$n$ is odd.

Let $G$ be a graph and $M$ its adjacency matrix, so that $M_{vu}$ is 1
if $(u,v) \in E$ and 0 otherwise.  We say that $G$ is a
$\lambda$-expander if the second-largest absolute eigenvalue of $M$ is
at most $\lambda$ (cf.~\cite{HLW:expanders}).  Expander graphs have
particularly nice properties under the iterated majority dynamics. One
reason for this is that in an expander graph, the number of edges
between disjoint sets $A$ and $B$ of vertices is almost completely
determined by the cardinalities of $A$ and $B$. We state this formally
in Lemma~\ref{lem:mixing} below.

Denote $E(A,B) = 1_A^T M 1_B$, where $A$ and $B$ be sets of vertices.
Note that if $A$ and $B$ are disjoint then $E(A,B)$ is the number
of edges between $A$ and $B$, and if $A$ and $B$ are not disjoint, then
$E(A,B)$ double-counts edges from $A \cap B$ to
itself). Alternatively, $E(A,B)$ is the number of ``edge-ends'' of
edges with one end in $A$ and another in $B$.

Recall that a graph $d$-regular if all vertices have degree $d$, i.e.,
$|N_v|=d$ for all $v \in V$.
\begin{lemma}[Expander mixing lemma (cf.~\cite{alon2008probabilistic})]\label{lem:mixing}
  If $G$ is a $d$-regular $\lambda$-expander with $n$ vertices then
  \begin{align*}
    \Big| E(A,B) - \frac{|A| |B| d}{n} \Big| \le \lambda \sqrt{|A| |B|}
  \end{align*}
  for every $A, B \subset G$.
\end{lemma}

It follows easily from the expander mixing lemma that medium-sized
majorities are unstable under iterated majority dynamics. That is,
if a reasonable majority of people prefer one outcome then very quickly
a large majority of people will prefer that outcome.

\begin{proposition}\label{prop:unstable}
  Let $q=2$, let $n$ be odd, let $G$ be a $d$-regular
  $\lambda$-expander with $d$ odd, and let the mode of interaction be
  majority dynamics with a majority vote at time $T$.

  Let $N_0(t)$ be the number of agents that vote $0$ at time $t$ and
  let $N_1(t)$ be the number of agents that vote $1$.  If $N_0(t)
  \ge N_1(t) + \alpha n$ then $N_1(t+1) \le \frac{2
    \lambda^2}{\alpha^2 d^2} n$.
\end{proposition}

\begin{proof}
  Let $A_0(t)$ be the set of agents that vote $0$ at time $t$, and
  define $A_1(t)$ similarly. Then, by the nature of majority dynamics,
  every $v \in A_1(t+1)$ has more than half of its neighbors in
  $A_1(t)$. Summing over every $v \in A_1(t+1)$, we have
  $E(A_1(t+1),A_1(t)) \ge E(A_1(t+1), A_0(t))$. By applying the
  expander mixing lemma to both sides,
  \begin{align*}
    \frac{N_1(t+1) N_0(t) d}{n} - \lambda \sqrt{N_1(t+1) N_0(t)} 
    \le \frac{N_1(t+1) N_1(t) d}{n} + \lambda \sqrt{N_1(t+1) N_1(t)}.    
  \end{align*}
  Rearranging, and since $N_0(t) - N_1(t) \ge \alpha n$,
  \begin{align*}
    \alpha \sqrt{N_1(t+1)} \le \frac{\lambda}{d} (\sqrt{N_1(t)} + \sqrt{N_0(t)}) \le \frac{\lambda}{d} \sqrt{2n}.    
  \end{align*}
\end{proof}

Applying the proposition twice, we see that an imbalance of $\frac{4
  \lambda n}{d}$ implies that a large, stable majority will form
within one time-step.

\begin{corollary}\label{cor:almost-consensus}
If $N_0(t) \ge N_1(t) + \frac{4 \lambda n}{d}$ and $\frac{\lambda}{d} \le \frac{3}{16}$ then
$N_1(s) \le \frac{n}{8}$ for all $s \ge t+1$.
\end{corollary}

\begin{proof}
Taking $\alpha = \frac{4 \lambda}{d}$ in Proposition~\ref{prop:unstable},
we have $N_1(t+1) \le \frac{n}{8}$. Then $N_0(t+1) \geq  N_1(t+1) + \frac{3n}{4}
\ge \frac{4 \lambda n}{d}$ and so we can continue applying Proposition~\ref{prop:unstable}
indefinitely with $\alpha = \frac{4 \lambda}{d}$.
\end{proof}

In order to show that a complete consensus is eventually achieved, we
will use a result of~\cite{GO:80}, who proved that majority dynamics
will eventually enter a cycle with period at most two.

\begin{proposition}\label{prop:consensus}
  If $\frac{\lambda}{d} \le \frac{3}{16}$ and $N_0(t) - N_1(t) \ge
  \frac{4 \lambda n}{d}$ for some $t$, then majority dynamics converge
  to all $0$.
\end{proposition}

\begin{proof}
Since majority dynamics converge to a cycle with period at most two,
we can divide the vertices of $G$ into four sets: $A_{00}$ is the set
of nodes that converge to $0$, $A_{11}$ is the set that converge to $1$,
with $A_{01}$ and $A_{10}$ being the two sets of nodes that eventually
alternate between $0$ and $1$. By Corollary~\ref{cor:almost-consensus},
$|A_{11}| + \max\{|A_{01}|, |A_{10}|\} \le \frac{n}{8}$, and so
$|A_{00}^c| = |A_{11}| + |A_{01}| + |A_{10}| \le \frac{n}{4}$.
By the expander mixing lemma,
\begin{align*}
|E(A_{00}^c, A_{00}^c)| \le |A_{00}^c|^2 \frac{d}{n} + \lambda |A_{00}^c|
\le |A_{00}^c|\left(\frac{d}{4} + \lambda\right).  
\end{align*}

On the other hand, $|E(A_{00}, A_{00}^c)| + |E(A_{00}^c, A_{00}^c)| =
d |A_{00}^c|$ and so $|E(A_{00}, A_{00}^c)| \ge |A_{00}^c|
(\frac{3d}{4} - \lambda)$. Since $\lambda \le d/4$, $|E(A_{00},
A_{00}^c)| \ge \frac{d}{2} |A_{00}^c|$.  Supposing that $A_{00}^c$ is
non-empty, there must be at least one vertex $v \in A_{00}^c$ with
more than half of its neighbors in $A_{00}$. But then the definition
of majority dynamics would imply that $v$ converges to $0$, a
contradiction. Thus $A_{00}^c$ must be empty, and all agents converge
to 0.
\end{proof}

In particular, a random $d$-regular graph has $\lambda = O(\sqrt d)$
with high probability.  Therefore, if we start with an initial bias
such that $\P(0) - \frac{1}{2} \gtrsim d^{-1/2}$ then iterated
majority on a random $d$-regular graph will converge to all $0$ with
high probability.

\subsection{Plurality dynamics on expanders}

The results of the previous section can be extended with little effort
to the case of more than two alternatives. The main obstacle in making
this extension is specifying the resolution of ties. With two
alternatives, we avoid the possibility of ties in majority dynamics
simply by requiring each vertex to have odd degree. With more than two
alternatives, the simplest way to avoid ties is to perturb the edge
weights slightly so that they are rationally independent.  Our
expansion assumptions can be easily extended to the weighted case: let
$M$ be the weighted adjacency matrix of $G$ and assume that all of its
entries on or above the main diagonal are rationally independent of
one another. Let $d$ be the largest absolute eigenvalue of $M$ and let
$\lambda$ be the second-largest.  Note that if $M$ was constructed by
perturbing the edge weights of a random regular graph, then $d$ will
be approximately the degree of the graph and $\lambda$ will be
$O(\sqrt d)$.

With the assumptions above, Lemma~\ref{lem:mixing} holds exactly as
it was stated above, and so the proof of Proposition~\ref{prop:unstable}
applies also.

\begin{proposition}\label{prop:unstable-q}
For $a \in [q]$, let $N_a(t)$ be the number of people that vote $a$
at time $t$. If $N_a(t) \ge \frac{1 + \alpha}{2} n$ then
$N_a(t+1) \ge n(1 - \frac{2\lambda^2}{\alpha^2 d^2})$
\end{proposition}

To get an extension of Proposition~\ref{prop:consensus}, we first
need to extend the periodicity result~\cite{GO:80} to the case of
several alternatives. This extension uses exactly the same argument
as~\cite{GO:80}, but we include it for completeness.

\begin{proposition}\label{prop:periodic-q}
  On a weighted graph with no ties, iterated plurality dynamics
  converge to a cycle of length at most 2.
\end{proposition}

\begin{proof}
  Consider the quantity
  \begin{align*}
    J_v(t) = \sum_{a \in [q]} \left((1_{\{X_v(t+1) = a\}} - 1_{\{X_v(t-1) = a\}}) \sum_{w \sim v} e_{wv} 1_{\{X_w(t) = a\}}\right),  
  \end{align*}
  where $e_{wv}$ is the weight of the edge between $v$ and $w$.
  Note that $J_v(t) \ge 0$ with equality if, and only if, $X_v(t+1) = X_v(t-1)$.
  Indeed, if $X_v(t+1) = X_v(t-1)$ then $J_v(t) = 0$ trivially, so suppose
  that $X_v(t+1) = a$ and $X_v(t-1) = b \ne a$.
  Then
  \begin{align*}
    J_v(t) = \sum_{\{w \sim v: X_w(t) = a\}} e_{wv} 
    - \sum_{\{w \sim v: X_w(t) = b\}} e_{wv}.
  \end{align*}
  Since $X_v(t+1) = a$ and the edge weights are chosen to ensure that
  ties never happen, this implies that $J_v(t) > 0$.
  
  Now consider $J(t) = \sum_v J_v(t)$. Note that if we define
  \begin{align*}
    L(t) = \sum_v \sum_{w \sim v} \sum_{a \in [q]} e_{wv} 1_{\{X_v(t+1) = a\}} 1_{\{X_w(t) = a\}}  
  \end{align*}
  then $J(t) = L(t) - L(t-1)$. Since the state space of the dynamics is
  finite and the dynamics are deterministic, the process eventually (by
  time $T$, say) converges to a cycle (of period $k$, say). Then
  \begin{align*}
    \sum_{t = T+1}^{T+k} J(t) = \sum_{t=T+1}^{T+k} L(t) -
    \sum_{t=T}^{T+k-1} L(t) = 0,
  \end{align*}
  since the states are identical at time $T$ and $T+k$, and thus $L(T) =
  L(T+k)$.  Since $J(t) \ge 0$ for every $t$, it follows that $J(T+1) =
  0$. Then $J_v(T+1) = 0$ for every $v$ and so the state at time $T+2$
  is identical to the state at time $T$.
\end{proof}

With Proposition~\ref{prop:periodic-q} in hand, the rest of the proof
of Proposition~\ref{prop:consensus} goes through in the $q$-alternative
case. We only note that we need to replace $A_{01}$ by the set
$A_{a*} = \{v : X_v(2t) = a \ne X_v(2t+1) \text{ for large enough } t\}$.

\begin{proposition}\label{prop:consensus-q}
If $\frac{\lambda}{d} \le \frac{3}{16}$ and $N_a(t) \ge n(\frac{1}{2} + \frac{2\lambda}{d})$
for some $t$ then the plurality dynamics converge to $a$.
\end{proposition}

In particular, if we take a random $d$-regular graph and perturb each edge
weight by at most $n^{-3}$, then the second eigenvalue will hardly change,
so we will still have $\lambda = O(\sqrt d)$. If $\P(X_v(0) = a) \ge
\P(X_v(0) = b) + \frac{c \sqrt {\log q}}{\sqrt d}$ for every $b \ne a$ then at time $t=1$,
with high probability most
of the vertices will prefer $a$ and Proposition~\ref{prop:consensus-q}
will imply that the plurality dynamics will converge to all $a$.

\subsection{A stronger result for expanders with large girth}
In Section~\ref{sec:majority-expanders} we proved that in majority
dynamics with two alternatives, an initial bias of $d^{-1/2}$ is
sufficient (on a random $d$-regular graph) for consensus in the
limit. Kanoria and Montanari~\cite{KM:trees} showed that on an
infinite $d$-regular tree, the required bias is much smaller as a
function of $d$:

\begin{theorem}[Kanoria and Montanari]\label{thm:mont-yash}
  Let $v$ be a vertex in an infinite $d$-regular tree. For any $\beta
  > 0$ and all sufficiently large $d$, if $\P(0) \ge \frac{1}{2} +
  d^{-\beta}$ then with probability one, $X_v(t) = 0$ for all
  sufficiently large $t$.
\end{theorem}

Using this, it is easy to improve our earlier bias requirement for
consensus from $\P(0) - \frac{1}{2} \gtrsim d^{-1/2}$ to $\P(0) -
\frac{1}{2} \gtrsim d^{-\beta}$ for any $\beta > 0$:

\begin{corollary}
  For every $d$, let $G_{n,d}$ be a sequence of $d$-regular
  $\lambda$-expanders with $\frac{\lambda}{d} \le \frac{3}{16}$, such
  that the girth of $G_{n,d}$ tends to infinity with $n$.  For any
  $\beta > 0$, if $p \ge \frac{1}{2} +d^{-\beta}$ then for all
  sufficiently large $d$, with high probability (as $n \to \infty$)
  the iterated majority process on $G_{n,d}$ will converge to all $0$.
\end{corollary}

\begin{proof}
  Choose $d$ large enough (depending on $\beta$) so that
  Theorem~\ref{thm:mont-yash} applies, then choose $T$ large enough so
  that $\P(X_v(T) = 0) \ge \frac{1}{2} + \frac{C}{\sqrt d}$ on the
  $d$-regular tree, for some constant $C$ to be determined.  By
  choosing $n$ large enough, we can ensure that the girth of $G_{n,d}$
  is larger than $T$; thus $\P(X_v(T) = 0) \ge \frac{1}{2} +
  \frac{C}{\sqrt d}$ for every $v \in G_{n,d}$. Then the expected
  fraction of nodes that are $0$ by time $T$ is at least $\frac{1}{2}
  + \frac{C}{\sqrt d}$, since at time $T$ each node only depends on
  the initial values of nodes within a ball of radius $T$. Since the
  number of such nodes is bounded as $n \to \infty$, McDiarmid's
  inequality~\cite{mcdiarmid1989method} implies that with high
  probability, at least $\frac{1}{2} + \frac{C - 1}{\sqrt d}$ fraction
  of nodes are $0$ at time $T$. If we choose $C$ large enough,
  Proposition~\ref{prop:consensus} implies that the dynamics converge
  to all $0$.
\end{proof}

\newpage
\FullVersionOnly{\bibliographystyle{abbrv} \bibliography{majority}}

\begin{thebibliography}{10}

\bibitem{nima}
N.~AhmadiPourAnari.
\newblock Unpublished manuscript, 2011.

\bibitem{alon2008probabilistic}
N.~Alon and J.~Spencer.
\newblock {\em The probabilistic method}, volume~73.
\newblock Wiley-Interscience, 2008.

\bibitem{bawa2003estimating}
M.~Bawa, H.~Garcia-Molina, A.~Gionis, and R.~Motwani.
\newblock Estimating aggregates on a peer-to-peer network.
\newblock {\em submitted for publication}, 2003.

\bibitem{berger2001dynamic}
E.~Berger.
\newblock Dynamic monopolies of constant size.
\newblock {\em Journal of Combinatorial Theory, Series B}, 83(2):191--200,
  2001.

\bibitem{MdC}
J.-A.-N. Condorcet.
\newblock {\em Essai sur l'application de l'analyse \`a la probabilit\'e des
  d\'ecisions rendues \`a la pluralit\'e des voix}.
\newblock De l'Imprimerie Royale, 1785.

\bibitem{DeGroot:74}
M.~H. DeGroot.
\newblock Reaching a consensus.
\newblock {\em Journal of the American Statistical Association},
  69(345):118--121, 1974.

\bibitem{fontes2002stretched}
L.~Fontes, R.~Schonmann, and V.~Sidoravicius.
\newblock Stretched exponential fixation in stochastic ising models at zero
  temperature.
\newblock {\em Communications in mathematical physics}, 228(3):495--518, 2002.

\bibitem{FK:96}
E.~Friedgut and G.~Kalai.
\newblock Every monotone graph property has a sharp threshold.
\newblock {\em Proceedings of the American Mathematical Society},
  124(10):2993--3002, 1996.

\bibitem{GO:80}
E.~Goles and J.~Olivos.
\newblock Periodic behaviour of generalized threshold functions.
\newblock {\em Discrete Mathematics}, 30(2):187--189, 1980.

\bibitem{golub2010naive}
B.~Golub and M.~Jackson.
\newblock Naive learning in social networks and the wisdom of crowds.
\newblock {\em American Economic Journal: Microeconomics}, 2(1):112--149, 2010.

\bibitem{HLW:expanders}
S.~Hoory, N.~Linial, and A.~Wigderson.
\newblock Expander graphs and their applications.
\newblock {\em Bulletin-American Mathematical Society}, 43(4):439--561, 2006.

\bibitem{howard2000zero}
C.~Howard.
\newblock Zero-temperature ising spin dynamics on the homogeneous tree of
  degree three.
\newblock {\em Journal of applied probability}, pages 736--747, 2000.

\bibitem{KaKaLi:88}
J.~Kahn, G.~Kalai, and N.~Linial.
\newblock The influence of variables on boolean functions.
\newblock In {\em Proceedings of the 29th Annual Symposium on Foundations of
  Computer Science}, pages 68--80, 1988.

\bibitem{Kalai:01}
G.~Kalai.
\newblock Social choice and threshold phenomena.
\newblock {\em Discussion Paper Series}, 2001.

\bibitem{Kalai:04}
G.~Kalai.
\newblock {Social Indeterminacy}.
\newblock {\em {Econometrica}}, 72:1565--1581, 2004.

\bibitem{KM:10}
G.~Kalai and E.~Mossel.
\newblock Sharp thresholds for non-boolean functions and social choice theory.
\newblock {\em Preprint}, 2010.

\bibitem{KM:trees}
Y.~Kanoria and A.~Montanari.
\newblock Majority dynamics on trees and the dynamic cavity method.
\newblock {\em Arxiv preprint arXiv:0907.0449}, 2009.

\bibitem{kempe2003gossip}
D.~Kempe, A.~Dobra, and J.~Gehrke.
\newblock Gossip-based computation of aggregate information.
\newblock In {\em Proceedings of the 44th Annual Symposium on Foundations of
  Computer Science}, pages 482--491. IEEE, 2003.

\bibitem{Margulis:74}
G.~Margulis.
\newblock Probabilistic characteristic of graphs with large connectivity.
\newblock {\em Problems Info. Transmission}, 10:174--179, 1977.

\bibitem{mcdiarmid1989method}
C.~McDiarmid.
\newblock On the method of bounded differences.
\newblock {\em Surveys in combinatorics}, 141(1):148--188, 1989.

\bibitem{MosselSlyTamuz11:arxiv}
E.~Mossel, A.~Sly, and O.~Tamuz.
\newblock From agreement to asymptotic learning.
\newblock Preprint at http://arxiv.org/abs/1105.4765, 2011.

\bibitem{Russo:81}
L.~Russo.
\newblock An approximate zero-one law.
\newblock {\em Probability Theory and Related Fields}, 61(1):129--139, 1982.

\bibitem{shah2009gossip}
D.~Shah.
\newblock Gossip algorithms.
\newblock {\em Foundations and Trends{\textregistered} in Networking},
  3(1):1--125, 2009.

\bibitem{Talagrand:94}
M.~Talagrand.
\newblock On {R}usso's approximate zero-one law.
\newblock {\em The Annals of Probability}, 22(3):1576--1587, 1994.

\end{thebibliography}
\end{document}